\documentclass[11pt]{article}

\usepackage{amsthm}
\usepackage{amsmath}
\usepackage{amssymb}
\usepackage{url}

\renewcommand{\a}{\alpha}
\newcommand{\chio}{\chi_{0*}}
\newcommand{\chis}{\chi_{s*}}
\newcommand{\chit}{\chi_{*t}}
\renewcommand{\d}{\delta}
\newcommand{\Fbm}{\reg{2}{m}}
\newcommand{\FbmE}{\Fbm E}
\newcommand{\FbmEhat}{\Fbmhat E}
\newcommand{\Fbmhat}{\reghat{2}{m}}
\newcommand{\Fm}{\reg{1}{m}}
\newcommand{\FmE}{\Fm E}
\newcommand{\FmEx}{\Fmx E}
\newcommand{\FmM}{\Fm M}
\newcommand{\Fmpi}{\Fm\pi}
\newcommand{\Fmrho}{\Fm\rho}
\newcommand{\Fmx}{\regx{1}{m}}
\newcommand{\g}{\gamma}
\newcommand{\gt}{\tilde{\g}}
\newcommand{\ghat}{\hat{\g}}
\newcommand{\GL}{\mathrm{GL}}
\newcommand{\id}{\mathrm{id}}
\newcommand{\jbar}{\bar{\jmath}}
\newcommand{\jbara}{\jbar^1}
\newcommand{\jbarataugt}{\jbara(\taumE\circ\gt)}
\newcommand{\Jbm}{J^2_m}
\newcommand{\JbmE}{\Jbm E}
\newcommand{\JbmEhat}{\Jbmhat E}
\newcommand{\Jbmhat}{\hat{J}^2_m}
\newcommand{\jbog}{j^2_0\g}
\newcommand{\jbophi}{j^2_0\phi}
\newcommand{\jbophihat}{j^2_0\phihat}
\newcommand{\jchio}{j^1_0\chi_{0*}}
\newcommand{\jchit}{j^1_0\chi_{*t}}
\newcommand{\jjchistart}{j^1_0(t\mapsto j^1_0\chit)}
\newcommand{\jjchit}{j^1_0(t\mapsto j^1_0\chi_t)}
\newcommand{\Jm}{J^1_m}
\newcommand{\JmE}{\Jm E}
\newcommand{\Jmpi}{\Jm\pi}
\newcommand{\joa}{j^1_0(1_{\R^m})}
\newcommand{\jofg}{j^1_0(f\circ\g)}
\newcommand{\jog}{j^1_0\g}
\newcommand{\jogab}{j^1_0(\g_1\cdot\g_2)}
\newcommand{\joghat}{j^1_0\ghat}
\newcommand{\jogphi}{j^1_0(\g\circ\phi)}
\newcommand{\jogt}{j^1_0\gt}
\newcommand{\jogtr}{j^1_0(\g\circ\tr_t)}
\newcommand{\joid}{j^1_0(\id_{\R^m})}
\newcommand{\jophi}{j^1_0\phi}
\newcommand{\jophihat}{j^1_0\phihat}
\newcommand{\jophiab}{j^1_0(\phi_1\circ\phi_2)}
\newcommand{\jopsix}{j^1_0(\psi\circ x^{-1})}
\newcommand{\jos}{j^1_0\sigma}
\newcommand{\jospinv}{j^1_0(\sigma\circ\phi^{-1})}
\newcommand{\Jpi}{J^1\pi}
\newcommand{\jppsi}{j^1_p\psi}
\newcommand{\Lbm}{L^2_m}
\newcommand{\Lbmhat}{\hat{L}^2_m}
\newcommand{\Lbmt}{\widetilde{L}^2_m}
\newcommand{\Lm}{L^1_m}
\renewcommand{\o}{\mathrm{o}}
\newcommand{\phihat}{\hat{\phi}}
\newcommand{\pimE}{\pi_{mE}}
\newcommand{\Pm}{P^{1,1}_m}
\newcommand{\R}{\mathbb{R}}
\newcommand{\reg}[2]{\raisebox{0ex}[1ex]{$\overset{\hspace{0.2em}\scriptscriptstyle\o}{%
  \raisebox{0ex}[1.3ex]{\phantom{T}}}\hspace{-0.65em}T^{#1}_{#2}$}}
\newcommand{\reghat}[2]{\raisebox{0ex}[2.5ex]{$\overset{\scriptscriptstyle\o}{%
  \raisebox{0ex}[1.8ex]{$\hat{\phantom{T}}$}}\hspace{-0.75em}T^{#1}_{#2}$}}
\newcommand{\regx}[2]{\overset{\hspace{0.2em}\scriptscriptstyle\o}{%
  \raisebox{0ex}[0.9ex]{$\scriptstyle\phantom{T}$}}\hspace{-0.52em}T^{#1}_{#2}}
\newcommand{\rhohat}{\hat{\rho}}  
\newcommand{\taumE}{\tau_{mE}}
\newcommand{\taumoE}{\raisebox{0ex}[1ex]{$\overset{\scriptscriptstyle\o}{%
  \raisebox{0ex}[0.7ex]{\phantom{$\tau$}}}\hspace{-0.6em}\tau_{mE}$}}
\newcommand{\taumTmE}{\tau_{m(\TmE)}}
\newcommand{\Tbm}{T^2_m}
\newcommand{\TbmE}{\Tbm E}
\newcommand{\TbmEhat}{\Tbmhat E}
\newcommand{\Tbmhat}{\hat{T}^2_m}
\newcommand{\Tm}{T^1_m}
\newcommand{\TmE}{\Tm E}
\newcommand{\Tmf}{\Tm f}
\newcommand{\TmM}{\Tm M}
\newcommand{\tr}{\mbox{\textsc{t}}}
\newcommand{\VmJmE}{V_m\JmE}
\newcommand{\VmJmEalt}{V_m^\wedge\JmE}
\newcommand{\VmJmEsym}{V_m^\vee\JmE}
\newcommand{\VmtaumE}{V_m\taumE}
\newcommand{\VmtaumEalt}{V_m^\wedge\taumE}
\newcommand{\VmtaumEsym}{V_m^\vee\taumE}
\newcommand{\VmtaumoE}{V_m\taumoE}
\newcommand{\VmtaumoEalt}{V_m^\wedge\taumoE}
\newcommand{\VmtaumoEsym}{V_m^\vee\taumoE}

\newcommand{\art}[6]{#1: #2 {\it #3\/} {\bf #4} (#5) #6}

\newcommand{\book}[4]{#1: {\it #2\/} (#3, #4)}

\newtheorem{thm}{Theorem}
\newtheorem{prop}[thm]{Proposition}
\newtheorem{lem}[thm]{Lemma}
\newtheorem{cor}[thm]{Corollary}

\title{Double structures and jets}
\author{D.J.~Saunders}
\date{}

\begin{document}

\maketitle

\begin{abstract}
We show how the double vector bundle structure of the manifold of double velocities, with its submanifolds of holonomic and semiholonomic double velocities, is mirrored by a structure of holonomic and semiholonomic subgroups in the principal prolongation of the first jet group. We use the actions of these groups to construct holonomic and semiholonomic submanifolds in the manifold of double contact elements, and show that these give rise to affine bundles where a semiholonomic element has well-defined holonomic and curvature components.\\[2ex]
\end{abstract}

\section{Introduction}

Since the time of Ehresmann it has been known that the \emph{velocities} on a manifold $E$, the equivalence classes of maps $\R^m\supset O\to E$ under the relation of having the same value and derivatives up to some order $r$ at zero, may be related to the \emph{contact elements}, the equivalence classes of $m$-dimensional submanifolds under the relation of having $r$-th order contact at a point; the relationship involves the action of jet groups. It has also been known that non-holonomic jets, where the action of taking a jet is repeated, give a generalisation which is important when considering questions of integrability, and that there is an intermediate notion of semiholonomic jets where the generalisation is concentrated on just the highest order derivatives, and which gives rise to a well-defined concept of curvature. These generalisations apply to both velocities and contact elements, and it is of interest to discover the relations between them. An article by~Kol\'{a}\v{r}~\cite{Kol11} (see also~\cite{KV10}) describes these relations, using the principal prolongations of jet groups~\cite{KMS93} to obtain nonholonomic contact elements from nonholonomic velocities.

At a meeting in Krak\'{o}w for the eightieth birthday of W.M.~Tulczyjew, the author gave a talk which also considered these relations in the special case of double $1$-velocities and double $1$-contact elements, and the present paper describes some elements of that talk. We show that the first principal prolongation of the first order jet group does indeed have a kind of `double structure' which is similar to that of the manifold of double $1$-velocities. Indeed, the group was described in the talk as the `double jet group', but Jean~Pradines has suggested to the author that this might result in confusion with Ehresmann's double groupoid, and so the terminology in this paper has been modified:\ we now call it the `principal jet group'. We show that this group has three distinguished subgroups, the first of which may be canonically identified with the second order (holonomic) jet group, the second of which has good claim to be called `semiholonomic', and the third of which may be called a `curvature' subgroup. Although the action of the whole group on the double velocity manifold is free only on a certain open submanifold, we show that the action of each of these subgroups is free on a larger submanifold which includes vertical velocities. In this way we are able to construct a split short exact sequence of vector bundles over the manifold of first order contact elements; we can therefore show that both the holonomic and the semiholonomic double contact elements form affine bundles, and that each semiholonomic element gives rise to a unique holonomic/curvature pair.

\section{Preliminaries}
In this section we give a rapid survey of basic results on velocities; the original ideas were described in the works of Ehresmann~\cite{Ehr83}, and~\cite{KMS93} is a useful reference. See also~\cite{Pra77} for double structures, and~\cite{Lib97} for semiholonomic jets.

Let $E$ be a manifold (finite-dimensional, smooth, Hausdorff, paracompact) with $\dim E = n$. If $m<n$ then the manifold of first order $m$-velocities in $E$ will be denoted by $\TmE$, and the open submanifold of regular velocities by $\FmE$. An element of $\TmE$ is therefore of the form $\jog$ where $\g:O\to E$ is a map from a connected open subset $O\subset \R^m$ with $0 \subset O$ into $E$; the velocity is regular when $\g$ is an immersion near zero. We shall let $\taumE:\TmE\to E$ be the projection, and $\taumoE:\FmE\to E$ its restriction to regular velocities. Any map $f:E_1\to E_2$ gives rise to a prolonged map $\Tmf:\TmE_1\to\TmE_2$ by $\Tmf(\jog)=\jofg$.

The jet group $\Lm$ has elements of the form $\jophi$ where $\phi:O\to\R^m$ satisfies $\phi(0)=0$ and is a diffeomorphism near zero; composition is $\jophi_1\circ\jophi_2=\jophiab$, and the group may be identified with $\GL(m,\R)$. This group has a right action on $\TmE$ given by $\jog\cdot\jophi = \jogphi$ which restricts to a free action on $\FmE$. The quotient by this free action is a Hausdorff manifold $\JmE$, the manifold of first order $m$-dimensional contact elements, and may be identified with the Grassmannian manifold of $m$-planes in $E$; we shall denote a typical element of $\JmE$ by $[\jog]$, and we shall let $\rho:\FmE\to\JmE$ be the projection and $\pimE:\JmE\to E$ the induced map.

If $E$ is fibred over some other manifold by $\pi:E\to M$ then $\Fmpi\subset\FmE$ will denote the subset of velocities where the composite $\pi\circ\g$ is an immersion near zero; $\Fmpi$ is an open submanifold of $\FmE$. The prolongation $\Tm\pi:\TmE\to\TmM$ restricts to $\Fmpi\to\FmM$, but need not restrict to $\FmE\to\FmM$.

The free action of $\Lm$ restricts to $\Fmpi$, and the quotient is an open submanifold $\Jmpi\subset\JmE$ which may be identified with the manifold $\Jpi$ of jets of local sections of $\pi$ by $\Jpi \to \Jmpi$, $\jppsi\mapsto[\jopsix]$ where $p\in M$ and where $x$ is a coordinate map on $M$ around $p$.

We may apply these constructions where we start, not with $E$ itself, but with its velocity manifold $\TmE$. The double velocity manifold $\Tm\TmE$ is the total space of a double vector bundle:\ that is, there are two commuting vector bundle structures $\taumTmE:\Tm\TmE\to\TmE$ and $\Tm\taumE:\Tm\TmE\to\TmE$. By considering maps $\chi:O\times O\to E$ we may express elements $\jogt$ of $\Tm\TmE$, where $\gt:O\to\TmE$, as double jets $\jjchit$ where $\chi_t(s)=\chi(s,t)$, and hence define an involution $e:\Tm\TmE\to\TmE$, the \emph{exchange map}, by sending $\chi(s,t)$ to $\chi(t,s)$. The involution links the two vector bundle structures on $\Tm\TmE$, because $\taumTmE\circ e = \Tm\taumE$.

The double vector bundle structure and the exchange map define distinguished submanifolds of $\Tm\TmE$. The \emph{holonomic submanifold} $\TbmE$ is the fixed point set of the exchange map, and the \emph{semiholonomic submanifold} $\TbmEhat$ is the subset where the two vector bundle projections $\taumTmE$ and $\Tm\taumE$ are equal; thus we see that \raisebox{0ex}[1ex]{$\TbmE\subset\TbmEhat$}.

These submanifolds may also be described by considering distinguished maps $O\to\TmE$. Given any map $\g:O\to E$, its prolongation is the map $\jbara\g:O\to\TmE$, $\jbara\g(t)=\jogtr$ where $\tr_t:\R^m\to\R^m$ is the translation $\tr_t(s)=s+t$. If $\gt:O\to\TmE$ is a prolongation then necessarily $\gt=\jbarataugt$; if this condition holds only at zero, so that $\gt(0)=\jbarataugt(0)$, then $\gt$ is called a \emph{semiprolongation}. A velocity $\jogt\in\Tm\TmE$ is semiholonomic when $\gt$ is a semiprolongation, and it is holonomic when some representative map $\ghat:O\to\TmE$, with $\joghat=\jogt$, is a prolongation (it need not be the case that $\gt$ itself is a prolongation). The holonomic submanifold may be identified with the set of second order $m$-velocities $\{\jbog\}$.

Both the holonomic and the semiholonomic submanifolds of $\Tm\TmE$ define affine sub-bundles of the vector bundle $\taumTmE:\Tm\TmE\to\TmE$. This may be seen by expressing the velocity bundle $\TmE$ as a Whitney sum $\bigoplus^m TE$ and hence as a tensor product $TE\otimes_E\R^{m*}$. The submanifold of vertical double velocities $\VmtaumE$ may be written as
\begin{align*}
\VmtaumE & \cong V\taumE\otimes_{\TmE}\R^{m*} \\
& \cong \taumE^*(\TmE)\otimes_{\TmE}\R^{m*} \\
& \textstyle \cong \taumE^*(TE)\otimes_{\TmE}\bigotimes^2\R^{m*} \, .
\end{align*}
This is the vector bundle over $\TmE$ on which $\TbmEhat$ is modelled; it may be decomposed into its symmmetric and skewsymmetric parts,
\[
\textstyle
\VmtaumEsym \cong \taumE^*(TE)\otimes_{\TmE} S^2\R^{m*} \, , \qquad
\VmtaumEalt \cong \taumE^*(TE)\otimes_{\TmE} \bigwedge^2\R^{m*} \, .
\]
The symmetric part is the vector bundle over $\TmE$ on which $\TbmE$ is modelled, and the decomposition gives rise to a decomposition of semiholonomic double velocities
\[
\TbmEhat = \TbmE \oplus_{\TmE} \VmtaumEalt
\]
into holonomic and `curvature' components.

The concept of regularity may be applied in several different ways to double velocities. One observation is that the the pair of maps $\Tm\TmE\to\TmE$ defining the double vector bundle structure restrict to a pair of maps $\Fm\taumoE\to\FmE$ rather than $\Fm\FmE\to\FmE$:\ that is, we need to consider regular velocities $\jogt$ satisfying the stronger condition that $\taumE\circ\g$ is an immersion. Also, we may define $\FbmE=\TbmE\cap\Tm\FmE$ and $\FbmEhat=\TbmEhat\cap\Tm\FmE$, and then
\[
\FbmE \subset \FbmEhat \subset \Fm\taumoE \, .
\]
If we restrict the vector bundles $\VmtaumE$, $\VmtaumEsym$ and $\VmtaumEalt$ to points of $\FmE$ then we still have the decompositions
\[
\VmtaumoE = \VmtaumoEsym \oplus_{\FmEx} \VmtaumoEalt
\]
and
\[
\FbmEhat = \FbmE \oplus_{\FmEx} \VmtaumoEalt \, .
\]

It is often convenient to use local coordinates for proofs. If $(U,u^a)$ is a chart on $E$ then coordinates on the preimage of $U$ in $\TmE$ are $(u^a, u^a_i)$ where $u^a_i(\jog) = D_i\g^a(0)$; similarly, coordinates on the preimage in $\Tm\TmE$ are $(u^a, u^a_i; u^a_{\cdot j}, u^a_{ij})$. The double vector bundle projections are given by
\begin{align*}
u^a\circ\taumTmE & = u^a \, , & u^a\circ\Tm\taumE & = u^a \, , \\
u^a_i\circ\taumTmE & = u^a_i \, , & u^a_i\circ\Tm\taumE & = u^a_{\cdot i}
\end{align*}
and the exchange map is given by
\begin{align*}
u^a\circ e & = u^a \, , & u^a_{\cdot j}\circ e & = u^a_j \\
u^a_i\circ e & = u^a_{\cdot i} \, , & u^a_{ij}\circ e & = u^a_{ji} \, .
\end{align*}
The semiholonomic submanifold is therefore defined by the constraint equations $u^a_i = u^a_{\cdot i}$, and the holonomic submanifold by these and the additional constraint equations $u^a_{ij} = u^a_{ji}$.

\section{The structure of the principal jet group}

The \emph{first order $m$-dimensional principal prolongation} of a Lie group $G$ is defined in~\cite{KMS93} as the semidirect product of $\Tm G$ and $\Lm$, where the velocity manifold $\Tm G$ is given the group operation $\jog_1 \cdot \jog_2 = \jogab$, with $\g_1\cdot\g_2$ being the pointwise product $(\g_1\cdot\g_2)(t) = \g_1(t)\g_2(t)$, and where the action of $\Lm$ on $\Tm G$ by automorphisms is the one described in the previous section. A similar definition is given for higher order principal prolongations. A recent paper by Kol\'{a}\v{r}~\cite{Kol11} has shown how to use these prolongations, in the case where $G$ is itself a jet group, to construct repeated contact elements from repeated velocities. We shall describe this in the $(1,1)$-order case.

Recall first that $\Lm$ acts on $\FmE$ to give $\JmE$ as a quotient, so that $\Tm\Lm$ acts on $\Tm\FmE$ to give $\Tm\JmE$ as a quotient. Then, separately, $\Lm$ acts on the submanifold $\Fm\JmE\subset\Tm\JmE$ to give $\Jm\JmE$; consequently elements of $\Fm\rho\subset\Tm\FmE$ project twice to elements of $\Jm\JmE$. We shall write a general element of $\Jm\JmE$ as $[[\jogt]]$.

Write $\Pm$ for the semidirect product $\Lm\rtimes\Tm\Lm$, with group operation
\[
(\jophi_1,\jos_1) \cdot (\jophi_2,\jos_2) = \bigl( \jophi_1\cdot\jophi_2, (\jos_1\cdot\jophi_2)\cdot\jos_2 \bigr) \, ;
\]
if $A^i_j$ denote the global $\GL(m,\R)$ coordinates on $\Lm$ and $(A^i_j, B^i_{jk})$ the corresponding global coordinates on $\Tm\Lm$ then the product is given by
\begin{align*}
A^i_j(\jophi_1\cdot\jophi_2) & = A^i_k(\jophi_1) A^k_j(\jophi_2) \, , \\
A^i_j((\jos_1\cdot\jophi_2)\cdot\jos_2) & = A^i_k(\jos_1) A^k_j(\jos_2) \, , \\
B^i_{jk}((\jos_1\cdot\jophi_2)\cdot\jos_2) & = B^i_{hl}(\jos_1) A^h_j(\jos_2) A^l_k(\jophi_2) + A^i_h(\jos_1) B^h_{jk}(\jos_2)
\end{align*}
\begin{prop}
For maps $\chi:O\times O\to\R^m$ define $\chis,\chit:O\to \R^m$ by
\[
\chis(t) = \chi(s,t) - \chi(s,0) \, , \qquad \chit(s) = \chi(s,t) - \chi(0,t) \, ;
\]
then each element of $\Pm$ may be written uniquely in the form
\[
(\jophi,\jos) = \bigl( \jchio \, , \jjchistart \bigr)
\]
where $\chis$, $\chit$ are both immersions near zero.
\end{prop}
\begin{proof}
Given $\chi$ satisfying the immersion condition, let $\phi=\chio$ and $\sigma(t)=\jchit$. Conversely, given $\jophi$ and $\jos$, define $\chi$ by
\[
\chi^i(s,t) = A^i_j(\jos) s^j + A^i_j(\jophi) t^j + B^i_{jk}(\jos) s^j t^k \, . 
\]
The result then follows from the fact that $\bigl( \jchio \, , \jjchistart \bigr)$ is determined by $D_1\chi(0,0)$, $D_2\chi(0,0)$ and $D_1 D_2\chi(0,0)$.
\end{proof}
\begin{cor}
The correspondence $\chi(s,t)\mapsto\chi(t,s)$ determines a well-defined exchange map $e:\Pm\to\Pm$ satisfying $\lambda\circ e=\mu$ where
\[
\lambda,\mu:\Pm\to\Lm \, , \qquad \lambda(\jophi,\jos)=\jophi \, , \quad \mu(\jophi,\jos)=\sigma(0) \, .
\] 
\end{cor}
\begin{proof}
Immediate in coordinates.
\end{proof}
We may now define an element $(\jophi,\jos)\in\Pm$ to be \emph{semiholonomic} if $\lambda(\jophi,\jos)=\mu(\jophi,\jos)$, and to be \emph{holonomic} if $e(\jophi,\jos)=(\jophi,\jos)$. It is clear from the definition that a semiholonomic element is of the form $(\sigma(0),\jos)$. In coordinates, a semiholonomic element $(\jophi,\jos)$ satisfies $A^i_j(\jophi)=A^i_j(\jos)$, and a holonomic element satisfies this condition and, in addition, $B^i_{jk}(\jos)=B^i_{kj}(\jos)$.
\begin{prop}
The set $\Lbmhat$ of semiholonomic elements of $\Pm$ is a closed Lie subgroup. The set $\Lbm$ of holonomic elements is also a closed Lie subgroup which may be identified with the second order jet group.
\end{prop}
\begin{proof}
The fact that both $\Lbmhat$ and $\Lbm$ are closed Lie subgroups follows straightforwardly from the coordinate conditions. If $\jbophi$ is an element of the second order jet group then
\[
\jbophi = j_0\bigl(t \mapsto j_0(\tr_{-\phi(t)}\circ\phi\circ\tr_t)\bigr)
\]
where $t \mapsto j_0(\tr_{-\phi(t)}\circ\phi\circ\tr_t)$ is a map $O\to\Lm$, so we may identify $\jbophi$ with $(\jophi,\jbophi)\in\Lbm\subset\Pm$. Conversely, given an element $(\jophi,\jos)\in\Lbm$ we may define $\phihat$ by
\[
\phihat^i(t)=A^i_j(\jos) t^j + \tfrac{1}{2} B^i_{jk}(\jos) t^j t^k
\]
and then $\jophi=\jophihat$ and $\jos=\jbophihat$.
\end{proof}
We have seen from the coordinate formul\ae\ that $\Lbm$ is the subgroup of $\Lbmhat$ satisfying the symmetry condition $B^i_{jk}=B^i_{kj}$. There is also a subgroup $\Lbmt$ satisfying the corresponding skew-symmetry condition, $B^i_{jk}+B^i_{kj}=0$, and we shall call this the \emph{curvature subgroup}. It may be defined abstractly by letting $\vee:\Lbmhat\to\Lbm$ denote the projection arising from the symmetrising map $\chi(s,t)\mapsto\frac{1}{2}\bigl(\chi(s,t)+\chi(t,s)\bigr)$ and putting $\Lbmt=\Lm\rtimes\ker\vee$, where we regard $\Lm\subset\Lbmhat$ by $\jophi\mapsto\bigl(\jophi,\joa\bigr)$. 

\section{The principal jet group and double contact elements}

We define the right action of the principal jet group $\Pm$ on the double velocity manifold $\Tm\TmE$ (see~\cite{Kol11}) by
\[
\jogt \cdot (\jophi,\jos) = (\jogt\cdot\jophi)\cdot\jos \, ;
\]
in coordinates, putting $\jogt_1=\jogt \cdot (\jophi,\jos)$, this is
\begin{align*}
u^a(\jogt_1) & = u^a(\jogt) \\
u^a_i(\jogt_1) & = u^a_h(\jogt) A^h_i(\jos) \\
u^a_{\cdot j}(\jogt_1) & = u^a_{\cdot k}(\jogt_1) A^k_j(\jophi) \\
u^a_{ij}(\jogt_1) & = u^a_{hk} A^h_i(\jos) A^k_j(\jophi) + u^a_h B^h_{ij}(\jos) \, .
\end{align*}
\begin{lem}
The action of $\Pm$ restricts to a free action on $\Fmrho\subset\Tm\FmE$.
\end{lem}
\begin{proof}
To show that the action of $\Pm$ restricts to $\Fmrho$, write
\[
(\jophi,\jos)=\bigl(\joid,\jospinv\bigr)\cdot\bigl(\jophi,\joa\bigr)
\]
and use properties of the separate actions of $\Lm$ and $\Tm\Lm$. A similar approach shows that if $\jogt\cdot(\jophi,\jos)=\jogt$ when $\jogt\in\Fmrho$ then $\jophi=\joid$ and $\jos=\joa$.
\end{proof}
It is clear that the action of $\Pm$ is equivalent to the two-stage process of acting on $\Fmrho$ first with $\Tm\Lm$ and then with $\Lm$ as described in the previous section, and that the quotient of $\Fmrho$ by $\Pm$ is again $\Jm\JmE$.

By choosing a subgroup of $\Pm$ we may find a larger submanifold of $\Tm\TmE$ on which the action is free.
\begin{lem}
The action of the semiholonomic subgroup $\Lbmhat\subset\Pm$ is free on $\Tm\FmE$.
\end{lem}
\begin{proof}
This follows easily from the coordinate formula. If
\[
u^a_i(\jogt) = u^a_h(\jogt) A^h_i(\jos)
\]
then $A^h_i(\jos)=\d^h_i$ because $\gt$ takes its values in $\FmE$ so that the matrix $u^a_h(\jogt)$ has maximal rank; the semiholonomic condition now implies $A^k_j(\jophi)=\d^k_j$ so that $B^h_{ij}=0$.
\end{proof}
A justification for the names `holonomic subgroup' and `semiholonomic subgroup' comes from the following.
\begin{prop}
An element $(\jophi,\jos)\in\Pm$ is in the semiholonomic subgroup $\Lbmhat$ if, and only if, it maps elements of $\FbmEhat$ to $\FbmEhat$; it is in the holonomic subgroup $\Lbm$ if, and only if, it maps elements of $\FbmE$ to $\FbmE$.
\end{prop}
\begin{proof}
These follow from the coordinate formula. The direct arguments are straightforward; for the converse, if there is at least one element $\jogt\in\FbmEhat$ such that $\jogt\cdot(\jophi,\jos)\in\FbmEhat$ then
\[
u^a_i(\jogt)=u^a_{\cdot i}(\jogt) \, , \qquad u^a_k(\jogt) A^k_i(\jos) = u^a_{\cdot k}(\jogt) A^k_i(\jophi)
\]
so that the maximal rank of $u^a_i(\jogt)$ shows that $A^k_i(\jos)=A^k_i(\jophi)$ and hence that $(\jophi,\jos)\in\Lbmhat$. A similar argument applies for $\FbmE$ and $\Lbm$.
\end{proof}

We now consider the three vector sub-bundles $\VmtaumoE, \VmtaumoEsym, \VmtaumoEalt$ of $\Tm\FmE$.
\begin{lem}
The free action of $\Lbmhat$ on $\Tm\FmE$ restricts to $\VmtaumoE$; similarly the free actions of the subgroups $\Lbm$ and $\Lbmt$ restrict to $\VmtaumoEsym$ and $\VmtaumoEalt$ respectively.
\end{lem}
\begin{proof}
This also follows easily from the coordinate formula. 
\end{proof}
\begin{thm}
\label{Tvert}
Put
\[
\VmJmE = \VmtaumoE / \Lbmhat \, , \quad \VmJmEsym = \VmtaumoEsym / \Lbm \, , \quad \VmJmEalt = \VmtaumoEalt / \Lbmt \, ;
\]
then each quotient space is a manifold, and is the total space of a vector bundle over $\JmE$.
\end{thm}
\begin{proof}
Let $\rhohat:\VmtaumoE\to\VmJmE$ be the quotient. Using coordinates $(u^a,u^a_i,u^a_{ij})$ on $\VmtaumoE$ with $u^a_{\cdot j}=0$ and with $\det(u^j_i) \ne 0$, define coordinates $(u^a,u^\a_i,v^\a_{ij})$ on the quotient (where $\a=m+1,\ldots,n$) by
\begin{align*}
u^a\circ\rhohat & = u^a \\
u^\a_i\circ\rhohat & = u^\a_h r^h_i \\
v^\a_{ij}\circ\rhohat & = u^\a_{hk} r^h_i r^k_j - u^\a_h r^h_k u^k_{pq} r^p_i r^q_j
\end{align*}
where $r^k_j u^j_i = \d^k_i$. One may check that the coordinates are well-defined and give smooth transition functions, thus providing a manifold structure on the quotient. One may also check, using these coordinates, that the vector space structure on the fibres of $\VmtaumoE$ projects to the quotient, and that the $v^\a_{ij}$ are linear fibre coordinates corresponding to linear local trivialisations, so that $\VmJmE\to\JmE$ is a vector bundle.

A similar approach may be used for the other two quotient spaces.
\end{proof}
\begin{cor}
The split short exact sequence of vector bundles over $\FmE$
\[
0\to\VmtaumoEsym\to\VmtaumoE\to\VmtaumoEalt\to 0
\]
corresponding to the direct sum decomposition of $\VmtaumoE$ projects to a split short exact sequence of vector bundles over $\JmE$
\[
0\to\VmJmEsym\to\VmJmE\to\VmJmEalt\to 0 \, .
\]
\end{cor}
\begin{proof}
Once again coordinates may be used to show that the injections
\[
\VmtaumoEsym,\VmtaumoEalt\to\VmtaumoE
\]
pass to the quotient. For instance, if $\jogt_1,\jogt_2\in\VmtaumoEsym$ project to the same element of $\VmJmEsym$ then
\[
u^a_{ij}(\jogt_1) = u^a_{hk}(\jogt_2) A^h_i A^k_j + u^a_k(\jogt_2) B^k_{ij}
\]
for some element of $\Lbmhat$ with coordinates $(A^k_j,B^k_{ij})$; but $u^a_{ij}(\jogt_1)=u^a_{ji}(\jogt_1)$ and $u^a_{hk}(\jogt_2)=u^a_{kh}(\jogt_2)$, so that the maximal rank of the matrix $u^a_k(\jogt_2)$ shows that $B^k_{ij}=B^k_{ji}$ so that the element of $\Lbmhat$ must in fact be contained in the subgroup $\Lbm$.
\end{proof}

Finally, we define an element $[[\jogt]]\in\Jm\JmE$ to be holonomic if at least one representative double velocity $\jogt$ is holonomic, and to be semiholonomic if at least one representative $\jogt$ is semiholonomic; we write $\JbmE$ for the subset of holonomic double contact elements, and $\JbmEhat$ for the subset of semiholonomic double contact elements.
\begin{thm}
The subset $\JbmEhat$ is a submanifold of $\Jm\JmE$ and is the total space of an affine bundle over $\JmE$ modelled on $\VmJmE$; the subset $\JbmE$ is also a submanifold of $\Jm\JmE$ and is the total space of an affine bundle over $\JmE$ modelled on $\VmJmEsym$. Furthermore, we may write
\[
\JbmEhat = \JbmE \oplus_{\JmE} \VmJmEalt \, ,
\]
giving a decomposition of semiholonomic double contact elements into holonomic and `curvature' components.
\end{thm}
\begin{proof}
Once again we use coordinates. The coordinates $(u^a,u^a_i,u^a_{ij})$ with $u^a_{\cdot i}=u^a_i$ and with $\det(u^j_i) \ne 0$ on $\FbmEhat$ may be used to define coordinates $(u^a,u^\a_i,v^\a_{ij})$ on $\JbmEhat$ in the same way as in Theorem~\ref{Tvert}; indeed the standard method of constructing coordinates on the manifold of first order contact elements may be used to construct coordinates $(u^a,u^\a_i,v^\a_{\cdot j},v^\a_{ij})$ on the open submanifold $\Jm\pimE\subset\Jm\JmE$, the quotient of $\Fm\taumE\subset\Fm\rho$ by $\Pm$, and then $\JbmEhat\subset\Jm\pimE$ is the submanifold satisfying the coordinate constraint $v^\a_{\cdot i}=u^\a_i$. We may then check that the affine action of $\VmtaumoE$ on $\FbmEhat$ projects to an affine action of the vector bundle $\VmJmE$ on $\JbmEhat$. The same argument applies in the holonomic case, where the additional constraint $u^a_{ij}=u^a_{ji}$ on $\FbmE$ gives rise to a similar constraint $v^\a_{ij}=v^\a_{ji}$ on $\JbmE$. The decomposition of semiholonomic double contact elements comes from the decomposition of semiholonomic double velocities and is independent of the choice of representative.
\end{proof}

\section*{Acknowledgements}

The author wishes to acknowledge the  support of grant no.\  201/09/0981 for Global Analysis and its Applications from the Czech Science Foundation, and also the joint IRSES project GEOMECH (EU FP7, nr 246981).

\bigskip\bigskip\bigskip
\noindent
Department of Mathematics, Faculty of Science\\
The University of Ostrava\\
30.\ dubna 22\\
701 03 Ostrava\\
Czech Republic

\bigskip
\noindent
Email: \url{david@symplectic.demon.co.uk}

\end{document}